\documentclass[12pt,a4paper]{amsart}
\allowdisplaybreaks[1]

\usepackage{amstext}
\usepackage{amsthm}
\usepackage{amssymb}
\usepackage{ytableau}
\usepackage{tikz}
\usepackage[all]{xy}
\usepackage[top=30truemm,bottom=30truemm,left=25truemm,right=25truemm]{geometry}

\DeclareFontEncoding{OT2}{}{}
\DeclareFontSubstitution{OT2}{cmr}{m}{l}

\DeclareFontFamily{OT2}{cmr}{\hyphenchar\font45}
\DeclareFontShape{OT2}{cmr}{m}{l}{%
<5><6><7><8><9>gen*wncyr%
<10><10.95><12><14.4><17.28><20.74><24.88>wncyr10}{}

\DeclareMathAlphabet{\mathcyr}{OT2}{cmr}{m}{l}

\newtheorem{thm}{Theorem}[section]
\newtheorem{lem}[thm]{Lemma}
\newtheorem{prop}[thm]{Proposition}

\theoremstyle{definition}
\newtheorem{defn}[thm]{Definition}

\theoremstyle{remark}
\newtheorem{rem}[thm]{Remark}

\newcommand{\sh}{\mathbin{\mathcyr{sh}}}
\newcommand{\CSF}{\mathrm{CSF}}
\newcommand{\bk}{\boldsymbol{k}}
\newcommand{\bl}{\boldsymbol{l}}
\newcommand{\bp}{\boldsymbol{p}}
\newcommand{\bQ}{\mathbb{Q}}
\newcommand{\bR}{\mathbb{R}}
\newcommand{\bZ}{\mathbb{Z}}

\newcommand{\cR}{\mathcal{R}}
\newcommand{\cS}{\mathcal{S}}
\newcommand{\cZ}{\mathcal{Z}}
\newcommand{\fH}{\mathfrak{H}}
\newcommand{\fP}{\mathfrak{P}}

\DeclareMathOperator{\wt}{wt}

\begin{document}

\title[$t$-adic SMZSV and CSF]{On variants of symmetric multiple zeta-star values and the cyclic sum formula}

\author{Minoru Hirose}
\address[Minoru Hirose]{Faculty of Mathematics, Kyushu University 744, Motooka, Nishi-ku,
Fukuoka, 819-0395, Japan} 
\email{m-hirose@math.kyushu-u.ac.jp}

\author{Hideki Murahara}
\address[Hideki Murahara]{Nakamura Gakuen University Graduate School,
 5-7-1, Befu, Jonan-ku, Fukuoka, 814-0198, Japan} 
\email{hmurahara@nakamura-u.ac.jp}

\author{Masataka Ono}
\address[Masataka Ono]{Multiple Zeta Research Center, Kyushu University, 744, Motooka, Nishi-ku,
Fukuoka, 819-0395, Japan} 
\email{m-ono@math.kyushu-u.ac.jp}

\thanks{This research was supported in part by JSPS KAKENHI Grant Numbers 16H06336, JP18J00982 and JP18K13392.}

\begin{abstract}
The $t$-adic symmetric multiple zeta values were defined by Jarossay, which have been studied as a real analogue of the $\boldsymbol{p}$-adic finite multiple zeta values.
In this paper, we consider the star analogues based on several regularization processes of multiple zeta-star values:
harmonic regularization, shuffle regularization, and Kaneko-Yamamoto's type regularization.

We also present the cyclic sum formula for $t$-adic symmetric multiple zeta(-star) values, which is the counterpart of that for $\boldsymbol{p}$-adic finite multiple zeta(-star) values obtained by Kawasaki. 
The proof uses our new relationship that connects the cyclic sum formula for $t$-adic symmetric multiple zeta-star values and that for the multiple zeta-star values.

\end{abstract}

\maketitle

\section{Introduction}
\subsection{Multiple zeta(-star) values and $t$-adic symmetric multiple zeta(-star) values}

For positive integers $k_1,\dots,k_r$ with $k_r \ge2$, the multiple zeta values (MZVs) and the multiple zeta-star values (MZSVs) are the real numbers defined by 
\begin{align*}
\zeta(k_1,\dots, k_r)
:=\sum_{1\le n_1<\cdots <n_r} \frac{1}{n_1^{k_1}\cdots n_r^{k_r}}, \quad
\zeta^\star (k_1,\dots, k_r):=\sum_{1\le n_1\le \cdots \le n_r} \frac{1}{n_1^{k_1}\cdots n_r^{k_r}}.
\end{align*} 
In the following, we call the tuple $(k_1, \ldots, k_r)$ of positive integers an index. 
We denote the index with $r=0$ by $\varnothing$ and we call it the empty index, and we understand $\zeta(\varnothing)=\zeta^{\star}(\varnothing)=1$. 
We also call an index $\bk=(k_1, \ldots, k_r)$ admissible if $\bk=\varnothing$ or $k_r\ge2$. 
For an index $\bk=(k_1, \ldots, k_r)$, the quantity $k_1+\cdots+k_r$ is called the weight of $\bk$ and we denote it by $\wt(\bk)$. 
Note that we set $\wt(\varnothing):=0$.

Let $\mathcal{Z}$ be the $\mathbb{Q}$-linear subspace of $\mathbb{R}$ spanned by $1$ and all multiple zeta values. 
Note that $\cZ$ becomes a $\bQ$-algebra.
For an index $(k_1,\dots,k_r)$, the $t$-adic symmetric multiple zeta values ($t$-adic SMZVs) are defined 
as elements in $ \cZ[[t]]$ by
\begin{align*}
\zeta^{\bullet}_{\widehat{\cS}}(k_1, \ldots, k_r)
:=&\sum_{i=0}^r(-1)^{k_{i+1}+\cdots+k_r}\zeta^{\bullet}(k_1, \ldots, k_i)\\
&\times \sum_{l_{i+1}, \ldots, l_r\ge0}\prod_{j=i+1}^r\binom{k_j+l_j-1}{l_j}
\zeta^{\bullet}(k_r+l_r, \ldots, k_{i+1}+l_{i+1})t^{l_{i+1}+\cdots+l_r}\\
&\quad\qquad\qquad\qquad\qquad\qquad\qquad\qquad\qquad\qquad\qquad\qquad\quad (\bullet \in \{\ast, \sh\}),
\end{align*}
which were introduced by Jarossay in \cite{Jar19} as a counterpart of $\bp$-adic finite multiple zeta values ($\bp$-adic FMZVs)
(see also \cite{Kan19} and \cite{KZ19} for ordinary FMZVs and SMZVs,
and see \cite{Ros15} and \cite{Sek17} for $\bp$-adic FMZVs). 
Here, the symbol $\zeta^\bullet$ on the right-hand side means the regularized values coming from the harmonic `$\ast$' or the shuffle `$\sh$' regularizations, 
i.e., real values obtained by taking constant terms of these regularizations 
as explained in \cite{IKZ06}.
Remark that $t$-adic SMZV is called $\Lambda$-adjoint multiple zeta values in \cite{Jar19}.

In this paper, we consider 3 types of $t$-adic symmetric multiple zeta-star values ($t$-adic SMZSVs) as elements in $\cZ[[t]]$ motivated by the previous works on MZSVs by Muneta \cite{Mun09}, Yamamoto \cite{Yam17}, and Kaneko-Yamamoto \cite{KY18}. 
For an index $(k_{1},\dots,k_{r})$, let
\[
\zeta^{\star, \bullet}(k_1,\dots, k_r)
:=\sum_{ \substack{\square \textrm{ is either a comma `,' } \\ \textrm{ or a plus `+'}} } 
\zeta^\bullet(k_1\square \cdots \square k_r)
\quad (\bullet \in \{\ast, \sh\}). 
\]
\begin{defn}[$t$-adic SMZSVs]
For an index $(k_{1},\dots,k_{r})$, let
\begin{align*}
\zeta_{\widehat{\cS}}^{\star,\bullet}(k_1,\dots,k_r)
&:=\sum_{i=0}^{r}(-1)^{k_{i+1}+\cdots+k_r}\zeta^{\star, \bullet}(k_1, \ldots, k_i)\\
&\times \sum_{l_{i+1}, \ldots, l_r\ge0}\prod_{j=i+1}^r\binom{k_j+l_j-1}{l_j}
\zeta^{\star, \bullet}(k_r+l_r, \ldots, k_{i+1}+l_{i+1})t^{l_{i+1}+\cdots+l_r}\\
&\quad\qquad\qquad\qquad\qquad\qquad\qquad\qquad\qquad\qquad\qquad\qquad\quad (\bullet \in \{\ast, \sh, KY\}),
\end{align*}
where the symbol $\zeta^{\star,KY}$ means the shuffle regularized values obtained by Yamamoto's integral representation of MZSVs (for precise definition, see Section 3.3).
\end{defn}
\begin{rem}
With a simple calculation, we have
\begin{align} \label{eq:tSMZSV by tSMZV}
\zeta^{\star, \bullet}_{\widehat{\cS}}(k_1, \ldots, k_r)=\sum_{ \substack{\square \textrm{ is either a comma `,' } \\ \textrm{ or a plus `+'}} } 
\zeta_{\widehat{\cS}}^\bullet(k_1\square \cdots \square k_r)
\end{align}
for $\bullet \in \{\ast, \sh\}$.
\end{rem}

\section{Main results}
\subsection{Relations among various zeta-star values}
The first result we introduce in this paper is the congruence between $\zeta^{\star, \sh}(\bk)$ and $\zeta^{\star, KY}(\bk)$.

\begin{thm}\label{main1}
For an index $\boldsymbol{k}$, we have
\begin{align*}
\zeta^{\star,\sh} (\boldsymbol{k})
&\equiv \zeta^{\star,KY} (\boldsymbol{k}) 
\pmod{\zeta(2)\cZ}.
\end{align*}
\end{thm}
By using this theorem, we see that 3 types of the $t$-adic SMZSVs are equivalent in $(\cZ/\zeta(2)\cZ)[[t]]$.
\begin{thm}\label{main2}
 For an index $\bk$, we have
 \begin{align*}
  \zeta^{\star,\ast}_{\widehat{\cS}}(\bk)
  \equiv\zeta^{\star,\sh}_{\widehat{\cS}} (\bk)
  \equiv\zeta^{\star,KY}_{\widehat{\cS}} (\bk)
 \end{align*}
 in $(\cZ/\zeta(2)\cZ)[[t]]$.
\end{thm}
In the sequel, thanks to Theorem \ref{main2}, we denote their $\bmod{\;\zeta(2)\cZ[[t]]}$ reduction by $\zeta^{\star}_{\widehat{\cS}}(\bk)$.

\subsection{Cyclic sum formulas}
Our second result is the cyclic sum formula for $\zeta^{\star, KY}_{\widehat{\cS}}(\bk)$. 
The original formulas for MZVs and MZSVs were, respectively, obtained by Hoffman and Ohno \cite{HO03}, and Ohno and Wakabayashi. For a positive integer $k$ and a non-negative integer $s$, we denote $\underbrace{k, \ldots, k}_{s}$ by $\{k\}^s$.
\begin{thm}[Cyclic sum formula; Ohno-Wakabayashi \cite{OW06}] \label{eq: CSF MZSV}
 Let $k$ be a positive integer.
 For a non-empty index $\bk=(k_1,\dots,k_r)$ with $\wt(\bk)=k$, we have
 \begin{align*}
  \sum_{i=1}^r\sum_{j=0}^{k_i-2} \zeta^\star(j+1, k_{i+1},\dots, k_r, k_1,\dots, k_{i-1}, k_i-j)
  &=k\zeta(k+1)-\delta_{\bk, (\{1\}^r)}\cdot k\zeta(k+1).
 \end{align*} 
 Here, $\delta_{\bk, (\{1\}^r)}$ is 1 if $\bk=(\{1\}^r)$ and 0 if $\bk \neq (\{1\}^r)$.
\end{thm}
Kawasaki proved the analogous formula for $\bp$-adic FMZ(S)Vs in \cite{Kaw19}. 
Here we introduce their counterpart for $t$-adic SMZ(S)Vs. 
\begin{thm}\label{cycShat}
 Let $k$ be a positive integer.
 For a non-empty index $\bk=(k_1,\dots,k_r)$ with $\wt(\bk)=k$, we have
 \begin{align*}
 &\sum_{i=1}^r \sum_{j=0}^{k_i-2} \zeta^{\star,KY}_{\widehat{\cS}} (j+1, k_{i+1}, \ldots, k_r, k_1, \ldots, k_{i-1}, k_i-j) \\
 &=\sum_{i=1}^r \sum_{j=0}^\infty \zeta^{\star,KY}_{\widehat{\cS}} (j+1, k_{i+1}, \ldots, k_r, k_1, \ldots, k_i)t^j\\
  & \quad +k\zeta^{\star,KY}_{\widehat{\cS}}(k+1)-\delta_{\bk, (\{1\}^r)}\cdot (1+(-1)^{k+1})k\zeta(k+1). 
 \end{align*}
\end{thm}
\begin{rem}
We shall give the cyclic sum formula for $\zeta_{\widehat{\cS}}(\bk)$ in Section 6. See Theorem \ref{CSF for tSMZV}.
\end{rem}
\begin{rem}
 Sato and the first-named author obtained the formula for refined SMZVs, which is another generalization of the SMZVs.
\end{rem}

Let $\fH:=\bQ\langle x, y \rangle$ be the non-commutative polynomial ring over $\bQ$ with two variables $x$ and $y$, and $\fH^1:=\bQ+y \fH \supset \fH^0:=\bQ+y\fH x$ be the $\bQ$-subalgebras of $\fH$. For a positive integer $k$ and an index $\bk$ with $\wt(\bk)=k$, let $w^{\star}(\bk)$ be the element in $\fH^1$ corresponding to Kaneko-Yamamoto's integral representation of $\zeta^{\star, KY}(\bk)$ and let $w^{\star}_{\widehat{\cS}}(\bk)$ be the element in $\fH^1[[t]]$ corresponding to that of $\zeta^{\star, KY}_{\widehat{\cS}}(\bk)$ (for the precise definitions, see Section 3.3). 
Set
\begin{align*}
 w^{\star}_{\CSF}(\bk)
 &:=\sum_{i=1}^r\sum_{j=0}^{k_i-2}w^{\star}(j+1, k_{i+1}, \ldots, k_r, k_1, \ldots, k_{i-1}, k_i-j)-kw^{\star}(k+1) \in \fH^0, \\
 w^{\star}_{\CSF, \widehat{\cS}}(\bk)
 &:=\sum_{i=1}^r\sum_{j=0}^{k_i-2}w^{\star}_{\widehat{\cS}}(j+1, k_{i+1}, \ldots, k_r, k_1, \ldots, k_{i-1}, k_i-j) \\
  &\quad -\sum_{i=1}^r\sum_{j=0}^\infty w^{\star}_{\widehat{\cS}}(j+1, k_{i+1}, \ldots, k_r, k_1, \ldots, k_i)t^j 
  -kw^{\star}_{\widehat{\cS}}(k+1) \in \fH^1[[t]].
\end{align*}
\begin{thm}\label{thm: second main thm}
 Let $\bk=(k_1, \ldots, k_r)$ be an index with $\wt(\bk)=k$. Then we have the following equality in $\fH^1[[t]]$: 
 \begin{align*}
  w^{\star}_{\CSF, \widehat{\cS}}(\bk)
  &=w^\star_{\CSF}(\bk)
   +(-1)^{k+1} \sum_{\bl=(l_1, \ldots, l_r) \in \bZ^r_{\ge0}} \prod_{j=1}^r\binom{k_j+l_j-1}{l_j}
   w^{\star}_\CSF(\overline{\bk+\bl}) t^{l_1+\cdots+l_r}.
 \end{align*}
 Here, we set $\overline{\bk+\bl}:=(k_r+l_r, \ldots, k_1+l_1)$.
\end{thm}
\begin{rem}
 By Theorem \ref{thm: second main thm}, we can find Theorem \ref{cycShat} easily (see Section 5). 
\end{rem}

The contents of this paper is as follows. 
In the next section, we introduce the algebraic setup of MZVs and MZSVs, and give the precise definitions of Muneta's and Kaneko--Yamamoto's regularized MZSVs. 
In Section 4, we prove the equivalence of the definitions of $t$-adic SMZSVs. 
In Section 5, we give the proofs of Theorems \ref{cycShat} and \ref{thm: second main thm}. 
In the final section, we prove the cyclic sum formula for the $t$-adic SMZV $\zeta_{\widehat{\cS}}(\bk)$.

\section{Preliminaries for the proofs}
\subsection{Algebraic setup of MZVs}
We introduce the algebraic setup of MZVs and MZSVs along with \cite{Hof97}.
Let $z_{k}:=yx^{k-1}$. Note that $\fH^1=\bQ\langle z_k \mid k\ge1 \rangle$.
We define the $\bQ$-linear map $Z \colon \fH^0 \rightarrow \bR$ 
by $Z(1):=1$ and $Z(z_{k_1} \cdots z_{k_r}):=\zeta(k_1, \ldots, k_r)$ for an admissible index $(k_1, \ldots, k_r)$. 

We define the harmonic product $\ast \colon \fH^1 \times \fH^1 \rightarrow \fH^1$ and the shuffle product $\sh \colon \fH \times \fH \rightarrow \fH$ inductively by the following rules:
\begin{itemize}
\item[(i)] $w\ast1=1\ast w=w$ for any $w \in \fH^1$.
\item[(ii)] $w_1z_k \ast w_2z_l =(w_1 \ast w_2z_l)z_k+(w_1z_k \ast w_2)z_l+(w_1 \ast w_2)z_{k+l}$ for any $w_1, w_2 \in \fH^1$ and $k, l \in \bZ_{\ge1}$, 
\item[(i')] $w \sh 1=1 \sh w=w$ for any $w \in \fH$.
\item[(ii')] $w_1u_1 \sh w_2u_2 =(w_1 \sh w_2u_2)u_1+(w_1u_1 \sh w_2)u_2$ for any $w_1, w_2 \in \fH$ and $u_1, u_2 \in \{x, y\}$.
\end{itemize}

It is known that the map $Z$ preserves the harmonic product $\ast$ and the shuffle product $\sh$.  
That is, we have 
\begin{align*}
Z(w \ast w')=Z(w \sh w')=Z(w)Z(w')
\end{align*}
for $w, w' \in \fH^0$ (See \cite{Hof97}). 

\subsection{Regularization of MZVs and MZSVs}

Along with \cite{IKZ06} and \cite{Mun09}, we introduce the harmonic regularized MZV $\zeta^{\ast}(\bk)$ and MZSV $\zeta^{\star, \ast}(\bk)$, also the shuffle regularized MZV $\zeta^{\sh}(\bk)$ and MZSV $\zeta^{\star, \sh}(\bk)$. For $\bullet \in \{\ast, \sh\}$, we denote the $\bQ$-algebra $(\fH^1, \bullet)$ and its $\bQ$-subalgebra $(\fH^0, \bullet)$ by $\fH^1_{\bullet}$ and $\fH^0_{\bullet}$, respectively. 
It is known that $\fH^1_{\bullet} \cong \fH^0_{\bullet}[y]$ as $\bQ$-algebra for $\bullet \in \{\ast, \sh\}$ (See \cite{Hof97} for $\bullet=\ast$ and \cite{Re93} for $\bullet =\sh$). 
Thus, there exists a unique $\bQ$-linear map $Z^{\bullet} \colon \fH^1_{\bullet} \rightarrow \bR[T]$ satisfying that $Z^{\bullet}(y)=T$, $Z^{\bullet}|_{\fH^0_{\bullet}}=Z$ and $Z^{\bullet}$ preserves the product $\bullet$ on $\fH^1$. 
We denote the image of $Z^{\bullet}$ at $w \in \fH^1$ by $Z^{\bullet}(w; T)$. It is known that $Z^{\bullet}(w; T) \in \cZ[T]$. 

For an index $\bk=(k_1, \ldots, k_r)$, we define the harmonic (resp.\ shuffle) regularized MZV $\zeta^{\ast}(\bk)$ (resp.\ $\zeta^{\sh}(\bk)$) by 
\begin{align*}
\zeta^{\ast}(\bk):=Z^{\ast}(z_{\bk}; 0) \quad (\text{resp.\ } \zeta^{\sh}(\bk):=Z^{\sh}(z_{\bk}; 0)).
\end{align*}
Here we set $z_{\bk}:=z_{k_1}\cdots z_{k_r}$ for an index $\bk=(k_1, \ldots, k_r)$. Let $\sigma \colon \fH \rightarrow \fH$ be the $\bQ$-automorphism defined by $\sigma(x):=x$ and $\sigma(y):=x+y$, and $S \colon \fH^1 \rightarrow \fH^1$ the $\bQ$-linear map defined by $S(1):=1$ and $S(yw):=y\sigma(w)$ for $w \in \fH$. For an index $\bk=(k_1, \ldots, k_r)$, Muneta's harmonic (resp.\ shuffle) regularized MZSV $\zeta^{\star, \ast}(\bk)$ (resp. $\zeta^{\star, \sh}(\bk)$) is defined by
\begin{align*}
\zeta^{\star, \ast}(\bk):=Z^{\ast}(S(z_{\bk}); 0)
\quad (\text{resp. } \zeta^{\star, \sh}(\bk):=Z^{\sh}(S(z_{\bk}); 0)).
\end{align*}

\subsection{Kaneko-Yamamoto's shuffle regularization of MZSVs}

We introduce Kaneko-Yamamoto's shuffle regularized MZSV $\zeta^{\star, KY}(\bk)$ along with \cite{KY18}. This was introduced by using the multiple integral on 2-posets introduced by Yamamoto in \cite{Yam17}. So, first we recall the definition of 2-posets and the integrals on them.

A 2-poset is a pair $(X, \delta_X)$ consisting of a finite partially ordered set $X=(X, \le)$ and a map $\delta_X : X \rightarrow \{x, y\} \subset \fH$ which is called a label map of $X$. A 2-poset $(X, \delta_X)$ is admissible if $\delta_X(p)=x$ for all maximal elements $p$ in $X$ and $\delta_X(q)=y$ for all minimal elements $q$ in $X$. 

For an admissible 2-poset $(X, \delta_X)$, let 
\begin{align*}
I(X):=\int_{\Delta_X}\prod_{p \in X}\omega_{\delta_X(p)}(t_p),
\end{align*}
where
\begin{align*}
\Delta_X=\{(t_p)_p \in [0, 1]^X \mid t_p<t_q \text{ if } p<q \,(p,q\in X)\}
\end{align*}
and
\begin{gather*}
\omega_x(t):=\frac{dt}{t}, \quad \omega_y(t):=\frac{dt}{1-t}.
\end{gather*}
Note that the integral $I(X)$ converges if and only if the 2-poset $X$ is admissible.

We use Hasse diagrams to indicate 2-posets, with vertices $\circ$ and $\bullet$ corresponding to $\delta_X(p)=x$ and $y$, respectively.

It is known that MZV and MZSV can be written as the integral on a certain 2-poset.
Indeed, for an index $\bk=(k_1, \ldots, k_r)$, set 
\begin{align}\label{eq: def of X Xstar}
X^\star(\bk):=
\begin{xy}
{(0,-4) \ar @{{*}-o} (4,0)}, 
{(4,0) \ar @{.o} (8,4)}, 
{(8,4) \ar @{-{*}} (12,-4)}, 
{(12,-4) \ar @{.} (14,-2)}, 
{(16,0) \ar @{.} (20,0)}, 
{(22,0) \ar @{.{*}} (24,-4)}, 
{(24,-4) \ar @{-{o}} (28,0)}, 
{(28,0) \ar @{.o} (32,4)}, 
{(32,4) \ar @{-{*}} (36,-4)}, 
{(36,-4) \ar @{-{o}} (40,0)}, 
{(40,0) \ar @{.o} (44,4)}, 
{(0,-3) \ar @/^2mm/ @{-}^{k_r} (7,4)}, 
{(24,-3) \ar @/^2mm/ @{-}^{k_2} (31,4)}, 
{(36,-3) \ar @/^2mm/ @{-}^{k_1} (43,4)}, 
\end{xy}.
\end{align}
Then, Yamamoto proved
\begin{gather*}
\zeta^\star(\bk)=I(X^\star(\bk))
\end{gather*}
for an admissible index $\bk$ (\cite[Corollary 1.3]{Yam17}).
Note that for the empty 2-poset denoted by $\varnothing$, we set $I(\varnothing):=1$. 

We recall the algebraic setup of 2-posets. Let $\fP$ be the $\bQ$-algebra generated by the isomorphism classes of 2-posets whose multiplication is given by the disjoint union of 2-posets and $\fP^0$ be the $\bQ$-subalgebra of $\fP$ generated by the isomorphism classes of admissible 2-posets. Then the integral $I(X)$ is regarded as the $\bQ$-algebra map $I \colon \fP^0 \rightarrow \bR$. 
Moreover, for a 2-poset $(X, \delta_X)$, set
\begin{align*}
W(X):=\sum_{\substack{f \colon X \rightarrow \{1, 2, \ldots, \#X\} \\ \textrm{order preserving bijections}}}
u_1 \cdots u_{\#X}, 
\end{align*}
where $u_i:=\delta_X(f^{-1}(i))$. Then, $W$ is the unique $\bQ$-algebra homomorphism $W \colon \fP \rightarrow \fH_{\sh}$ satisfying the following properties:
\begin{itemize}
\item[(W1)] If the 2-poset $X=\{p_1<\cdots<p_k\}$ is totally ordered, the identity $W(X)=\delta_X(p_1) \cdots \delta_X(p_k)$ holds.
\item[(W2)] If $a$ and $b$ are non-comparable elements of a 2-poset $X$, the identity $W(X)=W(X^a_b)+W(X^b_a)$ holds. Here, $X^a_b$ (resp. $X^b_a$) denotes the 2-poset which is obtained from $X$ by adjoining the relation $a>b$ (resp. $a<b$) (see \cite[Definition 2.2]{Yam17}).
\end{itemize} 
Then we have $W(\fP^0)=\fH^0$ and $I=Z\circ W \colon \fP^0 \rightarrow \bR$. Moreover, we have $W(X \sqcup Y)=W(X)\sh W(Y)$ for 2-posets $X$ and $Y$. Here, $X\sqcup Y$ denotes the disjoint union of the 2-posets $X$ and $Y$. 

For an index $\bk$, set $w^{\star}(\bk):=W(X^{\star}(\bk))$. We define Kaneko-Yamamoto's shuffle regularized MZSVs $\zeta^{\star, KY}(\bk)$ by
\begin{align*}
\zeta^{\star,KY}(\bk):=Z^{\sh}(w^{\star}(\bk); 0).
\end{align*}
Moreover, for an index $\bk=(k_1, \ldots, k_r)$, set
\begin{align*}
X^{\star}_{\widehat{\cS}}(\bk)
&:=\sum_{i=0}^{r}(-1)^{k_{i+1}+\cdots+k_r}X^{\star}(k_1, \ldots, k_i)\\
&\qquad \sqcup \sum_{l_{i+1}, \ldots, l_r \ge0}\prod_{j=i+1}^r \binom{k_j+l_j-1}{l_j}
X^{\star}(k_r+l_r, \ldots, k_{i+1}+l_{i+1})t^{l_{i+1}+\cdots+l_r} \in \fP[[t]].
\end{align*}
We extend $W$ and $Z^{\sh}$ to the natural $\bQ$-algebra homomorphisms $W \colon \fP[[t]] \rightarrow \fH_{\sh}[[t]]$ and $Z^{\sh} \colon \fH^1_{\sh}[[t]] \rightarrow \cZ[[t]]$, respectively. Then, if we set $w^{\star}_{\widehat{\cS}}(\bk):=W\bigl(X^{\star}_{\widehat{\cS}}(\bk)\bigr)$ for an index $\bk$, we have
\begin{align*}
\zeta^{\star, KY}_{\widehat{\cS}}(\bk)=Z^{\sh}(w^{\star}_{\widehat{\cS}}(\bk); 0).
\end{align*}

\section{Proofs of Theorems \ref{main1}\ and \ref{main2}} 
In this section, we prove Theorems \ref{main1}\ and \ref{main2}.

\begin{proof}[Proof of Theorem \ref{main1}]
Let $\rho, \rho^{\star} \colon \bR[T] \rightarrow \bR[T]$ be the $\bR$-linear maps defined by the equalities
\begin{align} \label{eq100}
\rho(e^{Tx})=A(x)e^{Tx}, \quad \rho^{\star}(e^{Tx})=A(-x)^{-1}e^{Tx}
\end{align}
in $\bR[T][[x]]$ on which $\rho$ and $\rho^{\star}$ act coefficientwise, where
\begin{align*}
 A(x):=\Gamma(1+x)e^{\gamma x}=\exp\left(\sum_{n=2}^\infty\frac{(-1)^n}{n}\zeta(n)x^n\right) \in \bR[[x]]
\end{align*}
and $\gamma$ is Euler's constant. Then, it is known that
\begin{align*}
Z^{\sh}(z_{\bk}; T)=\rho(Z^{\ast}(z_{\bk}; T))
\end{align*}
for an index $\bk$ (\cite[Theorem 1]{IKZ06}) and we have
\begin{align}\label{rho Mstar}
 Z^{\sh}(S(z_{\bk}); T)=\rho\bigl(Z^{\ast}(S(z_{\bk}); T)\bigr).
\end{align}
Moreover, it is known that
\begin{align} \label{rho KYstar}
 Z^{\sh}(w^{\star}(\bk); T)=\rho^{\star}\bigl(Z^{\ast}(S(z_{\bk}); T)\bigr)
\end{align}
for an index $\bk$ (\cite[Corollary 4.7]{KY18}). Thus, since $\rho$ is invertible, by \eqref{rho Mstar} and \eqref{rho KYstar}, we have
\begin{align} \label{rho Msh to KYsh}
Z^{\sh}(w^{\star}(\bk); T))=\rho^{\star} \bigl( \rho^{-1}(Z^{\sh}(S(z_{\bk}); T)) \bigr).
\end{align}
On the other hand, since
\begin{align*}
\rho^{\star}\bigl(\rho^{-1}(e^{Tx})\bigr)=A(x)^{-1}A(-x)^{-1}e^{Tx}=\Gamma(1+x)^{-1}\Gamma(1-x)^{-1}e^{Tx}=\frac{\sin \pi x}{\pi x}e^{Tx}
\end{align*}
and
\begin{align*}
\frac{\sin \pi x}{\pi x}=\sum_{n=0}^{\infty}(-1)^n\frac{\pi^{2n}}{(2n+1)!}x^{2n},
\end{align*}
we have
\begin{align} \label{rho e^Tx}
\rho^{\star}\bigl(\rho^{-1}(T^n)\bigr)\equiv T^n \pmod{\zeta(2)\cZ}
\end{align}
for any non-negative integer $n$. Therefore, by \eqref{rho Msh to KYsh} and \eqref{rho e^Tx}, we have
\begin{align*}
Z^{\sh}(w^{\star}(\bk); T)\equiv Z^{\sh}(S(z_{\bk}); T) \bmod{\zeta(2)\cZ[T]}
\end{align*}
for any index $\bk$. This completes the proof of Theorem \ref{main1} by substituting $T=0$.
\end{proof}

\begin{proof}[Proof of Theorem \ref{main2}]
The former congruence follows from \eqref{eq:tSMZSV by tSMZV} and the fact 
$\zeta^{\ast}_{\widehat{\cS}}(\bk)\equiv \zeta^{\sh}_{\widehat{\cS}}(\bk) \pmod{\zeta(2)\cZ[[t]]}$
for an index $\bk$ (\cite{OSY19}).
The latter congruence follows from Theorem \ref{main1}.
\end{proof}

\section{Cyclic sum formula for $t$-adic SMZSVs}

In this section, we give the proofs of Theorems \ref{cycShat} and \ref{thm: second main thm} by using the theory of Yamamoto integral. 


\subsection{Preliminary}
In this subsection, we state a proposition which leads to Theorem \ref{main2}. For an index $\bk$ with $\wt(\bk)=k$, set
\begin{align*}
\widetilde{w}^{\star}_{\CSF, \widehat{\cS}}(\bk)
&:=w^{\star}_{\CSF, \widehat{\cS}}(\bk)+kw^{\star}_{\widehat{\cS}}(k+1),\\
\widetilde{w}^{\star}_{\CSF}(\bk)
&:=w^{\star}_{\CSF}(\bk)+kw^{\star}(k+1).
\end{align*}

\begin{prop}\label{prop: key prop}
Let $\bk=(k_1, \ldots, k_r)$ be an index with $\wt(\bk)=k$. Then we have
\begin{align*}
\widetilde{w}^{\star}_{\CSF, \widehat{\cS}}(\bk)
=\widetilde{w}^\star_{\CSF}(\bk)
+(-1)^{k+1}\sum_{\bl=(l_1, \ldots, l_r)\in \bZ^r_{\ge0}}\prod_{j=1}^r\binom{k_j+l_j-1}{l_j}\widetilde{w}^{\star}_\CSF(\overline{\bk+\bl})t^{l_1+\cdots+l_r}.
\end{align*}
\end{prop}
For the proof of Proposition \ref{prop: key prop}, we introduce the cyclic equivalence classes of indices. For positive integers $k$ and $r$ with $r \le k$, set
\begin{align*}
I(k, r):=\{(k_1, \ldots, k_r) \in \bZ^r_{\ge1} \mid \wt(\bk)=k\}.
\end{align*}
We say two elements of $I(k, r)$ are cyclically equivalent if they are cyclic permutations of each other. That is, if we denote the cyclic permutation $(1 \cdots r)$ of length $r$ by $\tau$, and $(k_1, \ldots, k_r)$ and $(k'_1, \ldots, k'_r)$ are elements in $I(k, r)$, we denote $(k_1, \ldots, k_r) \equiv (k'_1, \ldots, k'_r)$ if there exists $j \in \{1, \ldots, r\}$ such that $k'_i=k_{\tau^j(i)}$ for all $1\le i \le r$. Let $\Pi(k, r)$ be a set of cyclic equivalence classes of $I(k, r)$. 
For any $\alpha \in \Pi(k, r)$, set
\begin{align*}
 \widetilde{w}^{\star}_{\CSF, \widehat{\cS}}(\alpha)
 &:=\sum_{(k_1, \ldots, k_r) \in \alpha}
  \sum_{j=0}^{k_r-2} 
  w^{\star}_{\widehat{\cS}}(j+1, k_1, \ldots, k_{r-1}, k_r-j) \\
  &\quad -\sum_{(k_1, \ldots, k_r) \in \alpha}
  \sum_{j=0}^{\infty} 
  w^{\star}_{\widehat{\cS}}(j+1, k_1, \ldots, k_r)t^j, \\
  \widetilde{w}^{\star}_{\CSF}(\alpha)
 &:=\sum_{(k_1, \ldots, k_r) \in \alpha}\sum_{j=0}^{k_r-2}w^{\star}(j+1, k_1, \ldots, k_{r-1}, k_r-j)
\end{align*}
and 
\begin{align*}
\widetilde{u}^{\star}_{\CSF}(\alpha; \bl)
:=\sum_{(k_1, \ldots, k_r) \in \alpha}
\prod_{s=1}^r\binom{k_s+l_s-1}{l_s}
\sum_{j=0}^{k_1+l_1-2}w^{\star}(j+1, k_r+l_r, \ldots, k_2+l_2, k_1+l_1-j)
\end{align*}
for an $\bl :=(l_1, \ldots, l_r) \in \bZ^r_{\ge0}$. 
Then, for the proof of Proposition \ref{prop: key prop}, 
it suffices to prove
\begin{align}\label{eq: key identity}
\widetilde{w}^{\star}_{\CSF, \widehat{\cS}}(\alpha)
=\widetilde{w}^{\star}_{\CSF}(\alpha)+(-1)^{k+1}\sum_{\bl=(l_1, \ldots, l_r) \in \bZ^r_{\ge0}}\widetilde{u}^{\star}_{\CSF}(\alpha; \bl)t^{l_1+\cdots+l_r}.
\end{align}

Now, we prove \eqref{eq: key identity}. For an index  $\bk=(k_1, \ldots, k_r)$, set
\begin{align*}
F(\bk)
:=\sum_{l_1, \ldots, l_r \ge0}
\left\{\prod_{j=1}^r(-1)^{k_j}\binom{k_j+l_j-1}{l_j}\right\}
w^{\star}(k_r+l_r, \ldots, k_1+l_1)t^{l_1+\cdots+l_r},
\end{align*}
and
\begin{align*}
A&:=\sum_{(k_1, \ldots, k_r) \in \alpha}\sum_{\substack{a+b=k_r-1 \\ a\ge0, b\ge1}}\sum_{i=0}^{r-1}
w^{\star}(1+a, k_1, \ldots, k_i)F(k_{i+1}, \ldots, k_{r-1}, 1+b),\\
B&:=\sum_{(k_1, \ldots, k_r) \in \alpha}\sum_{\substack{a+b=k_r-1 \\ a\ge0, b\ge1}}
w^{\star}(1+a, k_1, \ldots, k_{r-1}, 1+b),\\
C&:=\sum_{(k_1, \ldots, k_r) \in \alpha}\sum_{\substack{a+b=k_r-1 \\ a\ge0, b\ge1}}
F(1+a, k_1, \ldots, k_{r-1}, 1+b).
\end{align*}
Then we have
\begin{align*}
\sum_{(k_1, \ldots, k_r) \in \alpha}\sum_{\substack{a+b=k_r-1 \\ a\ge0, b\ge1}}
w^{\star}_{\widehat{\cS}}(1+a, k_1, \ldots, k_{r-1}, 1+b)
=A+B+C.
\end{align*}

\subsection{Proofs of Theorems \ref{cycShat}\ and \ref{thm: second main thm}}
In this subsection, we give the proofs of Theorems \ref{cycShat} and \ref{thm: second main thm}.
It is easy to see that 
\begin{align}\label{eq: calc of B}
B=\widetilde{w}^{\star}_{\CSF}(\alpha).
\end{align}
We calculate $A$ and $C$ by using the following equality
\begin{align}\label{eq: key id1}
W
\left(
\begin{xy}
{(-4,4) \ar @{o.o} (-2,0)},
{(-2,0) \ar @{-{*}} (0,-4)},
{(-4.5,3) \ar @/_1mm/ @{-}_{c} (-3,0)},
{(4,4) \ar @{o.o} (2,0)},
{(2,0) \ar @{-{*}} (0,-4)},
{(4.5,3) \ar @/^1mm/ @{-}^{d} (3,0)},
\end{xy}
\right)
=\binom{c+d}{c}
W
\left(
\begin{xy}
{(-4,4) \ar @{o.o} (-2,0)},
{(-2,0) \ar @{-{*}} (0,-4)},
{(-4.5,3) \ar @/_1mm/ @{-}_{c+d} (-3,0)},
\end{xy}
\right).
\end{align}
Here, $c$ and $d$ are non-negative integers. This equality easily follows from the definition of $W$.

\begin{lem}\label{lem: calc of A}
We have
\begin{align*}
A=&\sum_{(k_1, \ldots, k_r) \in \alpha}
\Biggl[
\sum_{l=0}^{\infty}
\Bigl\{
w^{\star}_{\widehat{\cS}}(1+l, k_1, \ldots, k_r)-F(1+l, k_1, \ldots, k_r)
\Bigr\}t^l
+F(k_1, \ldots, k_r, 1)
\Biggr].
\end{align*}
\end{lem}

\begin{proof}
First, from the definition of $F(\bk)$ and \eqref{eq: key id1}, we have
\begin{align*}
A
&=\sum_{(k_1, \ldots, k_r) \in \alpha}\sum_{\substack{a+b=k_r-1 \\ a\ge0, b\ge1}}\sum_{i=0}^{r-1}\sum_{l_{i+1}, \ldots, l_{r-1}, l \ge0}
(-1)^{1+b}\binom{b+l}{l}t^{l}\prod_{j=i+1}^{r-1}\binom{k_j+l_j-1}{l_j}(-1)^{k_j}t^{l_j}\\
&\quad\times 
W\left(
\begin{xy}
{(0,-4) \ar @{{*}-o} (4,0)}, 
{(4,0) \ar @{.o} (8,4)}, 
{(8,4) \ar @{-{*}} (12,-4)}, 
{(12,-4) \ar @{.} (14,-2)}, 
{(16,0) \ar @{.} (20,0)}, 
{(22,0) \ar @{.{*}} (24,-4)}, 
{(24,-4) \ar @{-{o}} (28,0)}, 
{(28,0) \ar @{.o} (32,4)}, 
{(32,4) \ar @{-{*}} (36,-4)}, 
{(36,-4) \ar @{-{o}} (40,0)}, 
{(40,0) \ar @{.o} (44,4)},
{(0,-3) \ar @/^2mm/ @{-}^{k_i} (7,4)}, 
{(24,-3) \ar @/^2mm/ @{-}^{k_1} (31,4)}, 
{(40,1) \ar @/^1mm/ @{-}^{a} (43,4)}, 
{(52,4) \ar @{o.o} (56,0)},
{(56,0) \ar @{-{*}} (60,-4)},
{(60,-4) \ar @{-o} (64,4)},
{(64,4) \ar @{.o} (68,0)},
{(68,0) \ar @{-{*}} (72,-4)},
{(72,-4) \ar @{.} (74,0)},
{(76,0) \ar @{.} (80,0)},
{(82,-2) \ar @{.} (84,-4)},
{(84,-4) \ar @{{*}-{o}} (88,4)},
{(88,4) \ar @{.{o}} (92,0)},
{(92,0) \ar @{-{*}} (96,-4)},
{(53,4) \ar @/^1mm/ @{-}^{b+l} (56,1)},
{(65,4) \ar @/^2mm/ @{-}^{k_{r-1}+l_{r-1}} (72,-3)}, 
{(89,4) \ar @/^2mm/ @{-}^{k_{i+1}+l_{i+1}} (96,-3)}, 
\end{xy}
\right)\\
&=\sum_{(k_1, \ldots, k_r) \in \alpha}\sum_{\substack{a+b=k_r-1 \\ a\ge0, b\ge1}}\sum_{i=0}^{r-1}\sum_{l_{i+1}, \ldots, l_{r-1}, l \ge0}
(-1)^{1+b}t^{l}\prod_{j=i+1}^{r-1}\binom{k_j+l_j-1}{l_j}(-1)^{k_j}t^{l_j}\\
&\quad \times
W\left(
\begin{xy}
{(0,-4) \ar @{{*}-o} (4,0)}, 
{(4,0) \ar @{.o} (8,4)}, 
{(8,4) \ar @{-{*}} (12,-4)}, 
{(12,-4) \ar @{.} (14,-2)}, 
{(16,0) \ar @{.} (20,0)}, 
{(22,0) \ar @{.{*}} (24,-4)}, 
{(24,-4) \ar @{-{o}} (28,0)}, 
{(28,0) \ar @{.o} (32,4)}, 
{(32,4) \ar @{-{*}} (36,-4)}, 
{(36,-4) \ar @{-{o}} (40,0)}, 
{(40,0) \ar @{.o} (44,4)},
{(0,-3) \ar @/^2mm/ @{-}^{k_i} (7,4)}, 
{(24,-3) \ar @/^2mm/ @{-}^{k_1} (31,4)}, 
{(40,1) \ar @/^1mm/ @{-}^{a} (43,4)}, 
{(52,4) \ar @{o.o} (56,0)},
{(56,0) \ar @{-{*}} (60,-4)},
{(60,-4) \ar @{-o} (64,4)},
{(64,4) \ar @{.o} (68,0)},
{(68,0) \ar @{-{*}} (72,-4)},
{(72,-4) \ar @{.} (74,0)},
{(76,0) \ar @{.} (80,0)},
{(82,-2) \ar @{.} (84,-4)},
{(84,-4) \ar @{{*}-{o}} (88,4)},
{(88,4) \ar @{.{o}} (92,0)},
{(92,0) \ar @{-{*}} (96,-4)},
{(52,3) \ar @/_1mm/ @{-}_{b} (55,0)},
{(65,4) \ar @/^2mm/ @{-}^{k_{r-1}+l_{r-1}} (72,-3)}, 
{(89,4) \ar @/^2mm/ @{-}^{k_{i+1}+l_{i+1}} (96,-3)},
{(60, -4) \ar @{{*}-o} (60, 0)},
{(60, 0) \ar @{o.o} (60, 8)},
{(59.5,0.5) \ar @/^1mm/ @{-}^{l} (59.5, 7.5)},
\end{xy}
\right).
\end{align*}
Since
\begin{align*}
&W\left(
\begin{xy}
{(0,-4) \ar @{{*}-o} (4,0)}, 
{(4,0) \ar @{.o} (8,4)}, 
{(8,4) \ar @{-{*}} (12,-4)}, 
{(12,-4) \ar @{.} (14,-2)}, 
{(16,0) \ar @{.} (20,0)}, 
{(22,0) \ar @{.{*}} (24,-4)}, 
{(24,-4) \ar @{-{o}} (28,0)}, 
{(28,0) \ar @{.o} (32,4)}, 
{(32,4) \ar @{-{*}} (36,-4)}, 
{(36,-4) \ar @{-{o}} (40,0)}, 
{(40,0) \ar @{.o} (44,4)},
{(0,-3) \ar @/^2mm/ @{-}^{k_i} (7,4)}, 
{(24,-3) \ar @/^2mm/ @{-}^{k_1} (31,4)}, 
{(40,1) \ar @/^1mm/ @{-}^{a} (43,4)}, 
{(52,4) \ar @{o.o} (56,0)},
{(56,0) \ar @{-{*}} (60,-4)},
{(60,-4) \ar @{-o} (64,4)},
{(64,4) \ar @{.o} (68,0)},
{(68,0) \ar @{-{*}} (72,-4)},
{(72,-4) \ar @{.} (74,0)},
{(76,0) \ar @{.} (80,0)},
{(82,-2) \ar @{.} (84,-4)},
{(84,-4) \ar @{{*}-{o}} (88,4)},
{(88,4) \ar @{.{o}} (92,0)},
{(92,0) \ar @{-{*}} (96,-4)},
{(52,3) \ar @/_1mm/ @{-}_{b} (55,0)},
{(65,4) \ar @/^2mm/ @{-}^{k_{r-1}+l_{r-1}} (72,-3)}, 
{(89,4) \ar @/^2mm/ @{-}^{k_{i+1}+l_{i+1}} (96,-3)},
{(60, -4) \ar @{{*}-o} (60, 0)},
{(60, 0) \ar @{o.o} (60, 8)},
{(59.5,0.5) \ar @/^1mm/ @{-}^{l} (59.5, 7.5)},
\end{xy}
\right)\\
&=W\left(
\begin{xy}
{(0,-4) \ar @{{*}-o} (4,0)}, 
{(4,0) \ar @{.o} (8,4)}, 
{(8,4) \ar @{-{*}} (12,-4)}, 
{(12,-4) \ar @{.} (14,-2)}, 
{(16,0) \ar @{.} (20,0)}, 
{(22,0) \ar @{.{*}} (24,-4)}, 
{(24,-4) \ar @{-{o}} (28,0)}, 
{(28,0) \ar @{.o} (32,4)}, 
{(32,4) \ar @{-{*}} (36,-4)}, 
{(36,-4) \ar @{-{o}} (40,0)}, 
{(40,0) \ar @{.o} (44,4)},
{(44,4) \ar @{-o} (48,8)},
{(0,-3) \ar @/^2mm/ @{-}^{k_i} (7,4)}, 
{(24,-3) \ar @/^2mm/ @{-}^{k_1} (31,4)}, 
{(40,1) \ar @/^1mm/ @{-}^{a} (43,4)}, 
{(48,8) \ar @{.o} (56,0)},
{(56,0) \ar @{-{*}} (60,-4)},
{(60,-4) \ar @{-o} (64,4)},
{(64,4) \ar @{.o} (68,0)},
{(68,0) \ar @{-{*}} (72,-4)},
{(72,-4) \ar @{.} (74,0)},
{(76,0) \ar @{.} (80,0)},
{(82,-2) \ar @{.} (84,-4)},
{(84,-4) \ar @{{*}-{o}} (88,4)},
{(88,4) \ar @{.{o}} (92,0)},
{(92,0) \ar @{-{*}} (96,-4)},
{(48,7) \ar @/_1mm/ @{-}_{b} (55,0)},
{(65,4) \ar @/^2mm/ @{-}^{k_{r-1}+l_{r-1}} (72,-3)}, 
{(89,4) \ar @/^2mm/ @{-}^{k_{i+1}+l_{i+1}} (96,-3)},
{(60, -4) \ar @{{*}-o} (60, 0)},
{(60, 0) \ar @{o.o} (60, 8)},
{(59.5,0.5) \ar @/^1mm/ @{-}^{l} (59.5, 7.5)},
\end{xy}
\right)\\
&\quad+W\left(
\begin{xy}
{(0,-4) \ar @{{*}-o} (4,0)}, 
{(4,0) \ar @{.o} (8,4)}, 
{(8,4) \ar @{-{*}} (12,-4)}, 
{(12,-4) \ar @{.} (14,-2)}, 
{(16,0) \ar @{.} (20,0)}, 
{(22,0) \ar @{.{*}} (24,-4)}, 
{(24,-4) \ar @{-{o}} (28,0)}, 
{(28,0) \ar @{.o} (32,4)}, 
{(32,4) \ar @{-{*}} (36,-4)}, 
{(36,-4) \ar @{-{o}} (40,0)}, 
{(40,0) \ar @{.o} (48,8)},
{(0,-3) \ar @/^2mm/ @{-}^{k_i} (7,4)}, 
{(24,-3) \ar @/^2mm/ @{-}^{k_1} (31,4)}, 
{(40,1) \ar @/^1mm/ @{-}^{a} (47,8)}, 
{(48,8) \ar @{-} (52,4)},
{(52,4) \ar @{o.o} (56,0)},
{(56,0) \ar @{-{*}} (60,-4)},
{(60,-4) \ar @{-o} (64,4)},
{(64,4) \ar @{.o} (68,0)},
{(68,0) \ar @{-{*}} (72,-4)},
{(72,-4) \ar @{.} (74,0)},
{(76,0) \ar @{.} (80,0)},
{(82,-2) \ar @{.} (84,-4)},
{(84,-4) \ar @{{*}-{o}} (88,4)},
{(88,4) \ar @{.{o}} (92,0)},
{(92,0) \ar @{-{*}} (96,-4)},
{(52,3) \ar @/_1mm/ @{-}_{b} (55,0)},
{(65,4) \ar @/^2mm/ @{-}^{k_{r-1}+l_{r-1}} (72,-3)}, 
{(89,4) \ar @/^2mm/ @{-}^{k_{i+1}+l_{i+1}} (96,-3)},
{(60, -4) \ar @{{*}-o} (60, 0)},
{(60, 0) \ar @{o.o} (60, 8)},
{(59.5,0.5) \ar @/^1mm/ @{-}^{l} (59.5, 7.5)},
\end{xy}
\right),
\end{align*}
we see that the sum $\sum_{\substack{a+b=k_r-1 \\ a\ge0, b \ge1}}(-1)^{b+1}$ is a telescoping sum. Thus we have

\begin{align}\label{eq: calc of A}
A&=\sum_{(k_1, \ldots, k_r) \in \alpha}\sum_{i=0}^{r-1}\sum_{l_{i+1}, \ldots, l_{r-1}, l \ge0}
(-1)^{k_r}t^{l}\prod_{j=i+1}^{r-1}\binom{k_j+l_j-1}{l_j}(-1)^{k_j}t^{l_j}\\
&\quad \times
W\left(
\begin{xy}
%
%
{(0,-4) \ar @{{*}-o} (4,0)}, 
{(4,0) \ar @{.o} (8,4)}, 
{(8,4) \ar @{-{*}} (12,-4)}, 
{(12,-4) \ar @{.} (14,-2)}, 
{(16,0) \ar @{.} (20,0)}, 
{(22,0) \ar @{.{*}} (24,-4)}, 
{(24,-4) \ar @{-{o}} (28,0)}, 
{(28,0) \ar @{.o} (32,4)}, 
{(0,-3) \ar @/^2mm/ @{-}^{k_i} (7,4)}, 
{(24,-3) \ar @/^2mm/ @{-}^{k_1} (31,4)}, 
%
%
{(32, 4) \ar @{-{*}} (39, 2)},
{(39,2) \ar @{-o} (46,0)},
{(46,0) \ar @{.o} (53,-2)},
{(53,-2) \ar @{-} (60,-4)},
{(46.5,-1) \ar @/_1mm/ @{-}_{k_r} (59,-4.5)},
%
%
{(60,-4) \ar @{-o} (64,4)},
{(64,4) \ar @{.o} (68,0)},
{(68,0) \ar @{-{*}} (72,-4)},
{(72,-4) \ar @{.} (74,0)},
{(76,0) \ar @{.} (80,0)},
{(82,-2) \ar @{.} (84,-4)},
{(84,-4) \ar @{{*}-{o}} (88,4)},
{(88,4) \ar @{.{o}} (92,0)},
{(92,0) \ar @{-{*}} (96,-4)},
{(65,4) \ar @/^2mm/ @{-}^{k_{r-1}+l_{r-1}} (72,-3)}, 
{(89,4) \ar @/^2mm/ @{-}^{k_{i+1}+l_{i+1}} (96,-3)},
%
%
{(60, -4) \ar @{{*}-o} (60, 0)},
{(60, 0) \ar @{o.o} (60, 8)},
{(59.5,0.5) \ar @/^1mm/ @{-}^{l} (59.5, 7.5)},
\end{xy}
\right) \nonumber \\
&\quad+\sum_{(k_1, \ldots, k_r) \in \alpha}\sum_{i=0}^{r-1}\sum_{l_{i+1}, \ldots, l_{r-1}, l \ge0}
(-1)^{2}t^{l}\prod_{j=i+1}^{r-1}\binom{k_j+l_j-1}{l_j}(-1)^{k_j}t^{l_j} \nonumber \\
&\quad \times
W\left(
\begin{xy}
%
%
{(0,-4) \ar @{{*}-o} (4,0)}, 
{(4,0) \ar @{.o} (8,4)}, 
{(8,4) \ar @{-{*}} (12,-4)}, 
{(12,-4) \ar @{.} (14,-2)}, 
{(16,0) \ar @{.} (20,0)}, 
{(22,0) \ar @{.{*}} (24,-4)}, 
{(24,-4) \ar @{-{o}} (28,0)}, 
{(28,0) \ar @{.o} (32,4)}, 
{(32,4) \ar @{-{*}} (36,-4)}, 
{(36,-4) \ar @{-{o}} (40,0)}, 
{(40,0) \ar @{.o} (44,4)},
{(0,-3) \ar @/^2mm/ @{-}^{k_i} (7,4)}, 
{(24,-3) \ar @/^2mm/ @{-}^{k_1} (31,4)}, 
{(36,-3) \ar @/^2mm/ @{-}^{k_r} (43,4)}, 
%
%
{(44,4) \ar @{-} (60,-4)},
{(60,-4) \ar @{-o} (64,4)},
{(64,4) \ar @{.o} (68,0)},
{(68,0) \ar @{-{*}} (72,-4)},
{(72,-4) \ar @{.} (74,0)},
{(76,0) \ar @{.} (80,0)},
{(82,-2) \ar @{.} (84,-4)},
{(84,-4) \ar @{{*}-{o}} (88,4)},
{(88,4) \ar @{.{o}} (92,0)},
{(92,0) \ar @{-{*}} (96,-4)},
{(65,4) \ar @/^2mm/ @{-}^{k_{r-1}+l_{r-1}} (72,-3)}, 
{(89,4) \ar @/^2mm/ @{-}^{k_{i+1}+l_{i+1}} (96,-3)},
{(60, -4) \ar @{{*}-o} (60, 0)},
{(60, 0) \ar @{o.o} (60, 8)},
{(59.5,0.5) \ar @/^1mm/ @{-}^{l} (59.5, 7.5)},
\end{xy}
\right) \nonumber \\
&=:A_1+A_2. \nonumber 
\end{align}

Next, we divide $A_2$ into $A_{21}$ and $A_{22}$.
From the property (W2) of $W$, we have
\begin{align}\label{eq: calc of A_2}
A_2=&\sum_{(k_1, \ldots, k_r) \in \alpha}\sum_{i=0}^{r-1}\sum_{l_{i+1}, \ldots, l_{r-1}, l\ge0}t^l\prod_{j=i+1}^{r-1}\binom{k_j+l_j-1}{l_j}(-1)^{k_j}t^{l_j}\\
&\times
W\left(
\begin{xy}
{(0,-4) \ar @{{*}-o} (4,0)}, 
{(4,0) \ar @{.o} (8,4)}, 
{(8,4) \ar @{-{*}} (12,-4)}, 
{(12,-4) \ar @{.} (14,-2)}, 
{(16,0) \ar @{.} (20,0)}, 
{(22,0) \ar @{.{*}} (24,-4)}, 
{(24,-4) \ar @{-{o}} (28,0)}, 
{(28,0) \ar @{.o} (32,4)}, 
{(32,4) \ar @{-{*}} (36,-4)}, 
{(36,-4) \ar @{-{o}} (40,0)}, 
{(40,0) \ar @{.o} (44,4)},
{(44,4) \ar @{o-{*}} (48,-4)},
{(48,-4) \ar @{-o} (52,0)},
{(52,0) \ar @{.o} (56,4)},
{(0,-3) \ar @/^2mm/ @{-}^{k_i} (7,4)}, 
{(24,-3) \ar @/^2mm/ @{-}^{k_1} (31,4)}, 
{(36,-3) \ar @/^1mm/ @{-}^{k_r} (43,4)}, 
{(52,1) \ar @/^1mm/ @{-}^{l} (55,4)},
{(64,4) \ar @{o.o} (68,0)},
{(68,0) \ar @{-{*}} (72,-4)},
{(72,-4) \ar @{.} (74,0)},
{(76,0) \ar @{.} (80,0)},
{(82,-2) \ar @{.} (84,-4)},
{(84,-4) \ar @{{*}-{o}} (88,4)},
{(88,4) \ar @{.{o}} (92,0)},
{(92,0) \ar @{-{*}} (96,-4)},
{(65,4) \ar @/^2mm/ @{-}^{k_{r-1}+l_{r-1}} (72,-3)}, 
{(89,4) \ar @/^2mm/ @{-}^{k_{i+1}+l_{i+1}} (96,-3)},
\end{xy}
\right) \nonumber \\
&-\sum_{(k_1, \ldots, k_r) \in \alpha}\sum_{i=0}^{r-2}\sum_{l_{i+1}, \ldots, l_{r-1}, l\ge0}t^l\prod_{j=i+1}^{r-1}\binom{k_j+l_j-1}{l_j}(-1)^{k_j}t^{l_j} \nonumber \\
&\times
W\left(
\begin{xy}
{(0,-4) \ar @{{*}-o} (4,0)}, 
{(4,0) \ar @{.o} (8,4)}, 
{(8,4) \ar @{-{*}} (12,-4)}, 
{(12,-4) \ar @{.} (14,-2)}, 
{(16,0) \ar @{.} (20,0)}, 
{(22,0) \ar @{.{*}} (24,-4)}, 
{(24,-4) \ar @{-{o}} (28,0)}, 
{(28,0) \ar @{.o} (32,4)}, 
{(32,4) \ar @{-{*}} (36,-4)}, 
{(36,-4) \ar @{-{o}} (40,0)}, 
{(40,0) \ar @{.o} (44,4)},
{(0,-3) \ar @/^2mm/ @{-}^{k_i} (7,4)}, 
{(24,-3) \ar @/^2mm/ @{-}^{k_1} (31,4)}, 
{(36,-3) \ar @/^1mm/ @{-}^{k_r} (43,4)}, 
{(44,4) \ar @{o-{*}} (48,2)},
{(48,2) \ar @{-o} (52,0)},
{(52,0) \ar @{.o} (56,-2)},
{(56,-2) \ar @{-{*}} (60,-4)},
{(60,-4) \ar @{-o} (64,4)},
{(64,4) \ar @{.o} (68,0)},
{(68,0) \ar @{-{*}} (72,-4)},
{(72,-4) \ar @{.} (74,0)},
{(76,0) \ar @{.} (80,0)},
{(82,-2) \ar @{.} (84,-4)},
{(84,-4) \ar @{{*}-{o}} (88,4)},
{(88,4) \ar @{.{o}} (92,0)},
{(92,0) \ar @{-{*}} (96,-4)},
{(52,-1) \ar @/_2mm/ @{-}_{k_{r-1}+l_{r-1}} (59,-4.5)}, 
{(65,4) \ar @/^2mm/ @{-}^{k_{r-2}+l_{r-2}} (72,-3)},
{(89,4) \ar @/^2mm/ @{-}^{k_{i+1}+l_{i+1}} (96,-3)},
{(48, 2) \ar @{-o} (52, 5)},
{(52, 5) \ar @{o.o} (56, 8)},
{(52,6) \ar @/^1mm/ @{-}^{l} (55,8)},
\end{xy}
\right) \nonumber \\
=:& A_{21}+A_{22}. \nonumber 
\end{align}
Moreover, from the definition of $F(k_1, \ldots, k_r)$, we see that 
\begin{align}\label{eq: calc of A_{21}}
A_{21}
=&\sum_{(k_1, \ldots, k_r) \in \alpha}\sum_{i=0}^{r-1}\sum_{l=0}^{\infty}
w^{\star}(1+l, k_r, k_1, \ldots, k_i)t^l \cdot F(k_{i+1}, \ldots, k_{r-1})\\
=&\sum_{(k_1, \ldots, k_r) \in \alpha}\sum_{l=0}^{\infty}
\bigl\{w^{\star}_{\widehat{\cS}}(1+l, k_r, k_1, \ldots, k_{r-1}) \nonumber \\
&\qquad\qquad\qquad-w^{\star}(1+l)\cdot F(k_r, k_1, \ldots, k_{r-1})-F(1+l, k_r, k_1, \ldots, k_{r-1})\bigr\}t^l \nonumber \\
=&\sum_{(k_1, \ldots, k_r) \in \alpha}\sum_{l=0}^{\infty}
\bigl\{w^{\star}_{\widehat{\cS}}(1+l, k_1, \ldots, k_r) \nonumber \\
&\qquad\qquad\qquad-w^{\star}(1+l) \cdot F(k_1, \ldots, k_r)-F(1+l, k_1, \ldots, k_r)\bigr\}t^l. \nonumber
\end{align}
Note that the last equality follows from that the considered sum is invariant under the cyclic permutation $\tau$.

Now, we calculate $A_{22}$. From the property (W2) of $W$, it is easy to see that
\begin{align}\label{eq: key id2}
\sum_{l', l'' \ge 0}\binom{k+l''-1}{l''}
W
\left(
\begin{xy}
{(-16,5) \ar @{o.o} (-12,3)},
{(-12,3) \ar @{-{*}} (-8,1)},
{(-8,1) \ar @{-o} (-4,-1)},
{(-4,-1) \ar@{.o} (0,-3)},
{(0,-3) \ar @{-{*}} (4,-5)},
{(-16,4) \ar @/_1mm/ @{-}_{l'} (-12.5,2.5)},
{(-4,-2) \ar @/_1mm/ @{-}_{k+l''} (3,-5)},
\end{xy}
\right)
t^{l'+l''}
=\sum_{l \ge 0}
W
\left(
\begin{xy}
{(-6,6) \ar @{{*}-} (-4,2)},
{(-4,2) \ar @{o.o} (-2,-2)},
{(-2,-2) \ar @{-{*}} (0,-6)},
{(-4,1) \ar @/_1mm/ @{-}_{k} (-1,-6)},
{(4,2) \ar @{o.o} (2,-2)},
{(2,-2) \ar @{-{*}} (0,-6)},
{(4.5,1) \ar @/^1mm/ @{-}^{l} (3,-2)},
\end{xy}
\right)
t^l
\end{align}
for a positive integer $k$. By using \eqref{eq: key id2}, we have
\begin{align*}
A_{22}
&=-\sum_{(k_1, \ldots, k_r)\in \alpha}\sum_{i=0}^{r-2}\sum_{l_{i+1}, \ldots, l_{r-2}, l \ge0}
(-1)^{k_{r-1}}t^l\prod_{j=i+1}^{r-2}\binom{k_j+l_j-1}{l_j}(-1)^{k_j}t^{l_j}\\
&\qquad \times 
W\left(
\begin{xy}
%
%
{(0,-4) \ar @{{*}-o} (4,0)}, 
{(4,0) \ar @{.o} (8,4)}, 
{(8,4) \ar @{-{*}} (12,-4)}, 
{(12,-4) \ar @{.} (14,-2)}, 
{(16,0) \ar @{.} (20,0)}, 
{(22,0) \ar @{.{*}} (24,-4)}, 
{(24,-4) \ar @{-{o}} (28,0)}, 
{(28,0) \ar @{.o} (32,4)}, 
{(32,4) \ar @{-{*}} (36,-4)}, 
{(36,-4) \ar @{-{o}} (40,0)}, 
{(40,0) \ar @{.o} (44,4)},
{(0,-3) \ar @/^2mm/ @{-}^{k_i} (7,4)}, 
{(24,-3) \ar @/^2mm/ @{-}^{k_1} (31,4)}, 
{(36,-3) \ar @/^1mm/ @{-}^{k_r} (43,4)}, 
%
%
{(44,4) \ar @{o-{*}} (48,2)},
{(48,2) \ar @{-o} (52,0)},
{(52,0) \ar @{.o} (56,-2)},
{(56,-2) \ar @{-{*}} (60,-4)},
{(52,-1) \ar @/_2mm/ @{-}_{k_{r-1}} (59,-4.5)}, 
%
%
{(60,-4) \ar @{-o} (64,4)},
{(64,4) \ar @{.o} (68,0)},
{(68,0) \ar @{-{*}} (72,-4)},
{(72,-4) \ar @{.} (74,0)},
{(76,0) \ar @{.} (80,0)},
{(82,-2) \ar @{.} (84,-4)},
{(84,-4) \ar @{{*}-{o}} (88,4)},
{(88,4) \ar @{.{o}} (92,0)},
{(92,0) \ar @{-{*}} (96,-4)},
{(65,4) \ar @/^2mm/ @{-}^{k_{r-2}+l_{r-2}} (72,-3)},
{(89,4) \ar @/^2mm/ @{-}^{k_{i+1}+l_{i+1}} (96,-3)},
%
%
{(60,-4) \ar @{-o} (60, 0)},
{(60, 0) \ar @{o.o} (60, 8)},
{(59,0) \ar @/^1mm/ @{-}^{l} (59, 8)},
\end{xy}
\right)\\
&=-\sum_{(k_1, \ldots, k_r)\in \alpha}\sum_{i=0}^{r-2}\sum_{l_{i+2}, \ldots, l_{r-1}, l \ge0}
(-1)^{k_r}t^l\prod_{j=i+2}^{r-1}\binom{k_j+l_j-1}{l_j}(-1)^{k_j}t^{l_j} \nonumber \\
&\qquad \times 
W\left(
\begin{xy}
%
%
{(0,-4) \ar @{{*}-o} (4,0)}, 
{(4,0) \ar @{.o} (8,4)}, 
{(8,4) \ar @{-{*}} (12,-4)}, 
{(12,-4) \ar @{.} (14,-2)}, 
{(16,0) \ar @{.} (20,0)}, 
{(22,0) \ar @{.{*}} (24,-4)}, 
{(24,-4) \ar @{-{o}} (28,0)}, 
{(28,0) \ar @{.o} (32,4)}, 
{(32,4) \ar @{-{*}} (36,-4)}, 
{(36,-4) \ar @{-{o}} (40,0)}, 
{(40,0) \ar @{.o} (44,4)},
{(0,-3) \ar @/^2mm/ @{-}^{k_{i+1}} (7,4)}, 
{(24,-3) \ar @/^2mm/ @{-}^{k_2} (31,4)}, 
{(36,-3) \ar @/^1mm/ @{-}^{k_1} (43,4)}, 
%
%
{(44,4) \ar @{o-{*}} (48,2)},
{(48,2) \ar @{-o} (52,0)},
{(52,0) \ar @{.o} (56,-2)},
{(56,-2) \ar @{-{*}} (60,-4)},
{(52,-1) \ar @/_2mm/ @{-}_{k_r} (59,-4.5)}, 
%
%
{(60,-4) \ar @{-o} (64,4)},
{(64,4) \ar @{.o} (68,0)},
{(68,0) \ar @{-{*}} (72,-4)},
{(72,-4) \ar @{.} (74,0)},
{(76,0) \ar @{.} (80,0)},
{(82,-2) \ar @{.} (84,-4)},
{(84,-4) \ar @{{*}-{o}} (88,4)},
{(88,4) \ar @{.{o}} (92,0)},
{(92,0) \ar @{-{*}} (96,-4)},
{(65,4) \ar @/^2mm/ @{-}^{k_{r-1}+l_{r-1}} (72,-3)},
{(89,4) \ar @/^2mm/ @{-}^{k_{i+2}+l_{i+2}} (96,-3)},
%
%
{(60,-4) \ar @{-o} (60, 0)},
{(60, 0) \ar @{o.o} (60, 8)},
{(59,0) \ar @/^1mm/ @{-}^{l} (59, 8)},
\end{xy}
\right). 
\end{align*}
The second identity above follows from that $A_{22}$ is invariant under the cyclic permutation $\tau$.
By changing the variable $i$ to $i+1$, we have
\begin{align} \label{eq: calc of A_22 2}
A_{22}&=-\sum_{(k_1, \ldots, k_r)\in \alpha}\sum_{i=1}^{r-1}\sum_{l_{i+1}, \ldots, l_{r-1}, l \ge0}
(-1)^{k_r}t^l\prod_{j=i+1}^{r-1}\binom{k_j+l_j-1}{l_j}(-1)^{k_j}t^{l_j}\\
&\qquad \times 
W\left(
\begin{xy}
%
%
{(0,-4) \ar @{{*}-o} (4,0)}, 
{(4,0) \ar @{.o} (8,4)}, 
{(8,4) \ar @{-{*}} (12,-4)}, 
{(12,-4) \ar @{.} (14,-2)}, 
{(16,0) \ar @{.} (20,0)}, 
{(22,0) \ar @{.{*}} (24,-4)}, 
{(24,-4) \ar @{-{o}} (28,0)}, 
{(28,0) \ar @{.o} (32,4)}, 
{(32,4) \ar @{-{*}} (36,-4)}, 
{(36,-4) \ar @{-{o}} (40,0)}, 
{(40,0) \ar @{.o} (44,4)},
{(0,-3) \ar @/^2mm/ @{-}^{k_{i}} (7,4)}, 
{(24,-3) \ar @/^2mm/ @{-}^{k_2} (31,4)}, 
{(36,-3) \ar @/^1mm/ @{-}^{k_1} (43,4)}, 
%
%
{(44,4) \ar @{o-{*}} (48,2)},
{(48,2) \ar @{-o} (52,0)},
{(52,0) \ar @{.o} (56,-2)},
{(56,-2) \ar @{-{*}} (60,-4)},
{(52,-1) \ar @/_2mm/ @{-}_{k_r} (59,-4.5)}, 
%
%
{(60,-4) \ar @{-o} (64,4)},
{(64,4) \ar @{.o} (68,0)},
{(68,0) \ar @{-{*}} (72,-4)},
{(72,-4) \ar @{.} (74,0)},
{(76,0) \ar @{.} (80,0)},
{(82,-2) \ar @{.} (84,-4)},
{(84,-4) \ar @{{*}-{o}} (88,4)},
{(88,4) \ar @{.{o}} (92,0)},
{(92,0) \ar @{-{*}} (96,-4)},
{(65,4) \ar @/^2mm/ @{-}^{k_{r-1}+l_{r-1}} (72,-3)},
{(89,4) \ar @/^2mm/ @{-}^{k_{i+1}+l_{i+1}} (96,-3)},
%
%
{(60,-4) \ar @{-o} (60, 0)},
{(60, 0) \ar @{o.o} (60, 8)},
{(59,0) \ar @/^1mm/ @{-}^{l} (59, 8)},
\end{xy}
\right). \nonumber 
\end{align}
Then, by comparing the definition of $A_1$ and \eqref{eq: calc of A_22 2} and using \eqref{eq: key id2}, we have
\begin{align}
A_{22}&=-A_1+\sum_{(k_1, \ldots, k_r)\in \alpha}\sum_{l_1, \ldots, l_{r-1}, l \ge0}
(-1)^{k_r}t^l\prod_{j=1}^{r-1}\binom{k_j+l_j-1}{l_j}(-1)^{k_j}t^{l_j} \nonumber \\
&\qquad \qquad \qquad \qquad \qquad \qquad \qquad \qquad \qquad \times 
W\left(
\begin{xy}
%
{(48,2) \ar @{{*}-o} (52,0)},
{(52,0) \ar @{.o} (56,-2)},
{(56,-2) \ar @{-{*}} (60,-4)},
{(52,-1) \ar @/_2mm/ @{-}_{k_r} (59,-4.5)}, 
%
%
{(60,-4) \ar @{-o} (64,4)},
{(64,4) \ar @{.o} (68,0)},
{(68,0) \ar @{-{*}} (72,-4)},
{(72,-4) \ar @{.} (74,0)},
{(76,0) \ar @{.} (80,0)},
{(82,-2) \ar @{.} (84,-4)},
{(84,-4) \ar @{{*}-{o}} (88,4)},
{(88,4) \ar @{.{o}} (92,0)},
{(92,0) \ar @{-{*}} (96,-4)},
{(65,4) \ar @/^2mm/ @{-}^{k_{r-1}+l_{r-1}} (72,-3)},
{(89,4) \ar @/^2mm/ @{-}^{k_1+l_1} (96,-3)},
%
%
{(60,-4) \ar @{-o} (60, 0)},
{(60, 0) \ar @{o.o} (60, 8)},
{(59,0) \ar @/^1mm/ @{-}^{l} (59, 8)},
\end{xy}
\right) \nonumber \\
&=-A_1+\sum_{(k_1, \ldots, k_r) \in \alpha}\sum_{l_1, \ldots, l_r, l \ge0}t^l \prod_{j=1}^{r}\binom{k_j+l_j-1}{l_j}(-1)^{k_j}t^{l_j} \nonumber \\
&\qquad \qquad \qquad \qquad \qquad \qquad \qquad \qquad \times
W\left(
\begin{xy}
%
%
{(40,6) \ar @{o.} (44, 4)},
{(44, 4) \ar @{o-} (48, 2)},
{(40,5) \ar @/_1mm/ @{-}_{l} (43.5, 3.5)},
%
%
{(48,2) \ar @{{*}-o} (52,0)},
{(52,0) \ar @{.o} (56,-2)},
{(56,-2) \ar @{-{*}} (60,-4)},
{(52,-1) \ar @/_2mm/ @{-}_{k_r+l_r} (59,-4.5)}, 
%
%
{(60,-4) \ar @{-o} (64,4)},
{(64,4) \ar @{.o} (68,0)},
{(68,0) \ar @{-{*}} (72,-4)},
{(72,-4) \ar @{.} (74,0)},
{(76,0) \ar @{.} (80,0)},
{(82,-2) \ar @{.} (84,-4)},
{(84,-4) \ar @{{*}-{o}} (88,4)},
{(88,4) \ar @{.{o}} (92,0)},
{(92,0) \ar @{-{*}} (96,-4)},
{(65,4) \ar @/^2mm/ @{-}^{k_{r-1}+l_{r-1}} (72,-3)},
{(89,4) \ar @/^2mm/ @{-}^{k_1+l_1} (96,-3)},
\end{xy}
\right). \nonumber
\end{align}
Finally, from the definition of $X^{\star}(\bk)$, we have
\begin{align}\label{eq: calc of A_{22}}
A_{22}&=-A_1+\sum_{(k_1, \ldots, k_r) \in \alpha} \sum_{l_1, \ldots, l_r, l \ge 0}
t^l\prod_{j=1}^r\binom{k_j+l_j-1}{l_j}(-1)^{k_j}t^{l_j} \\
&\qquad \qquad \times\left\{w^{\star}(1+l)\sh w^{\star}(k_r+l_r, \ldots, k_1+l_1)-w^{\star}(1+l, k_r+l_r, \ldots, k_1+l_1)\right\} \nonumber \\
&=-A_1+\sum_{(k_1, \ldots, k_r) \in \alpha}\sum_{l=0}^{\infty}
w^{\star}(1+l)t^l \cdot F(k_1, \ldots, k_r)+\sum_{(k_1, \ldots, k_r) \in \alpha}F(k_1, \ldots, k_r, 1). \nonumber 
\end{align}
From \eqref{eq: calc of A}, \eqref{eq: calc of A_2}, \eqref{eq: calc of A_{21}}, and \eqref{eq: calc of A_{22}}, we obtain the desired formula.
\end{proof}

Next, we calculate $C$.

\begin{lem}\label{lem: calc of C}
We have
\begin{align*}
C&=(-1)^{k+1}\sum_{\bl=(l_1, \ldots, l_r) \in \bZ^r_{\ge0}}
\widetilde{u}^{\star}_{\CSF}(\alpha; \bl)t^{l_1+\cdots+l_r}\\
&\quad + \sum_{(k_1, \ldots, k_r) \in \alpha}\sum_{l=0}^{\infty}F(1+l, k_1, \ldots, k_r)t^l
-\sum_{(k_1, \ldots, k_r) \in \alpha}F(k_1, \ldots, k_r, 1).
\end{align*}
\end{lem}

\begin{proof}
From the definition of $F(\bk)$, we have
\begin{align*}
C
&=\sum_{(k_1, \ldots, k_r) \in \alpha}\sum_{\substack{a+b=k_r-1 \\ a\ge0, b\ge0}}\sum_{l_1, \ldots, l_{r-1}, a', b' \ge0}
\binom{a+a'}{a}\binom{b+b'}{b}(-1)^{k_r+1}t^{a'+b'}\\
&\qquad \times\prod_{j=1}^{r-1}\binom{k_j+l_j-1}{l_j}(-1)^{k_j}t^{l_j}
\cdot w^{\star}(1+b+b', k_{r-1}+l_{r-1}, \ldots, k_1+l_1, 1+a+a') \\
&\qquad -\sum_{(k_1, \ldots, k_r) \in \alpha}F(k_1, \ldots, k_r, 1)\\
&=\sum_{(k_1, \ldots, k_r) \in \alpha}
\sum_{\substack{a+b=k_r-1 \\ a \ge 0, b \ge 0}}
\sum_{\substack{l_1, \ldots, l_{r-1} \ge0 \\ a' \ge a, b' \ge b}}
\binom{a'}{a}\binom{b'}{b}(-1)^{k_r+1}t^{a'+b'-k_r+1}\prod_{j=1}^{r-1}\binom{k_j+l_j-1}{l_j}(-1)^{k_j}t^{l_j}\\
&\qquad \times w^{\star}(1+b', k_{r-1}+l_{r-1}, \ldots, k_1+l_1, 1+a')-\sum_{(k_1, \ldots, k_r) \in \alpha}F(k_1, \ldots, k_r, 1).
\end{align*}
Set $l_r:=a'+b'-(k_r-1)$. Then, by Chu-Vandermonde identity, we have
\begin{align*}
&\sum_{\substack{a+b=k_r-1 \\ a \ge 0, b \ge 0}}\sum_{a' \ge a, b' \ge b}
\binom{a'}{a} \binom{b'}{b}t^{a'+b'-(k_r-1)}
w^{\star}(1+b', k_{r-1}+l_{r-1}, \ldots, k_1+l_1, 1+a')\\
&=\sum_{l_r \ge 0}\sum_{a'=0}^{k_r+l_r-1}\sum_{a=0}^{k_r-1}
\binom{a'}{a}\binom{k_r+l_r-1-a'}{k_r-1-a}t^{l_r}
w^{\star}(k_r+l_r-a', k_{r-1}+l_{r-1}, \ldots, k_1+l_1, 1+a')\\
&=\sum_{l_r \ge 0}\sum_{a'=0}^{k_r+l_r-1}
\binom{k_r+l_r-1}{l_r}t^{l_r}
w^{\star}(k_r+l_r-a', k_{r-1}+l_{r-1}, \ldots, k_1+l_1, 1+a')\\
&=\sum_{l_r \ge 0}\sum_{\substack{a'+b'=k_r+l_r-1 \\ a' \ge 0, b' \ge 0}}
\binom{k_r+l_r-1}{l_r}t^{l_r}
w^{\star}(1+b', k_{r-1}+l_{r-1}, \ldots, k_1+l_1, 1+a').
\end{align*}
Therefore, we have
\begin{align*}
C&=-\sum_{(k_1, \ldots, k_r) \in \alpha}
\sum_{l_1, \ldots, l_r \ge0}
\sum_{\substack{a'+b'=k_r-1+l_r \\ a' \ge 0, b' \ge 0}}
\prod_{j=1}^{r}\binom{k_j+l_j-1}{l_j}(-1)^{k_j}t^{l_j}\\
&\qquad \times w^{\star}(1+b', k_{r-1}+l_{r-1}, \ldots, k_1+l_1, 1+a')-\sum_{(k_1, \ldots, k_r) \in \alpha}F(k_1, \ldots, k_r, 1).
\end{align*}
Moreover, since 
\begin{align*}
w^{\star}(k_r+l_r, \ldots, k_1+l_1, 1)
=\sum_{l, l_0 \ge 0}(-1)^{l}\binom{l+l_0}{l_0}t^{l+l_0}
w^{\star}(k_r+l_r, \ldots, k_1+l_1, 1+l+l_0),
\end{align*}
we obtain
\begin{align*}
C&=(-1)^{k+1}\sum_{\bl=(l_1, \ldots, l_r) \in \bZ^r_{\ge0}}
\widetilde{u}^{\star}_{\CSF}(\alpha; \bl)t^{l_1+\cdots+l_r}\\
& \qquad + \sum_{(k_1, \ldots, k_r) \in \alpha}\sum_{l=0}^{\infty}F(1+l, k_1, \ldots, k_r)t^l-\sum_{(k_1, \ldots, k_r) \in \alpha}F(k_1, \ldots, k_r, 1),
\end{align*}
which completes the proof.
\end{proof}

\begin{proof}[Proof of Proposition \ref{prop: key prop}]
From \eqref{eq: calc of B}, Lemmas \ref{lem: calc of A} and \ref{lem: calc of C}, we obtain the formula \eqref{eq: key identity}, which leads to Proposition \ref{prop: key prop}.
\end{proof}

\begin{proof}[Proof of Theorem \ref{thm: second main thm}]
From Proposition \ref{prop: key prop}, it suffices to prove that
\begin{align*}
kw^{\star}_{\widehat{\cS}}(k+1)
=kw^{\star}(k+1)+(-1)^{k+1}\sum_{l_1, \ldots, l_r \ge0}(k+l)
\prod_{j=1}^r \binom{k_j+l_j-1}{l_j}w^{\star}(k+l+1)t^{l},
\end{align*}
where $l:=l_1+\cdots+l_r$. This equality follows immediately from the fact that $w^{\star}(n)=z_n=yx^{n-1}$ for a positive integer $n$ and an easy calculation.
\end{proof}

\begin{proof}[Proof of Theorem \ref{cycShat}]
Theorem \ref{cycShat} follows immediately from applying $Z^{\sh}$ to the both hand sides of Theorem \ref{thm: second main thm}, substituting $T=0$, and using the cyclic sum formula for MZSVs \cite{OW06}. Note that $w^{\star}_{\CSF}(\{1\}^k)=-kw^{\star}(k+1)=-kz_{k+1}$. 
\end{proof}


\section{Cyclic sum formula for $t$-adic SMZVs}
In this section, we prove the cyclic sum formula for $\zeta_{\widehat{\cS}}(\bk)$ from that for $\zeta^{\star}_{\widehat{\cS}}(\bk)$.
\begin{thm}\label{CSF for tSMZV}
For a non-empty index $\bk=(k_1, \ldots, k_r)$, we have
 \begin{align*}
 &\sum_{i=1}^r \sum_{j=0}^{k_i-2} \zeta_{\widehat{\cS}} (j+1, k_{i+1}, \ldots, k_r, k_1, \ldots, k_{i-1}, k_i-j) \\
 &=\sum_{i=1}^r \sum_{j=0}^\infty 
  \left(
   \zeta_{\widehat{\cS}} (k_i+j+1, k_{i+1},  \ldots, k_r, k_1, \ldots, k_{i-1})
   +\zeta_{\widehat{\cS}} (j+1, k_{i+1}, \ldots, k_r, k_1, \ldots, k_i) 
  \right) t^j \\
  &\quad +\sum_{i=1}^r \zeta_{\widehat{\cS}} (k_{i+1}, \ldots, k_r, k_1, \ldots,k_{i-1},  k_i+1).
\end{align*}
\end{thm}

\subsection{Preliminary}
Let $\cR$ be the $\bQ$-vector space generated by all the indices. For an index $\bk=(k_1, \ldots ,k_r)$, set
\begin{align*}
\bk^{\star}:=\sum_{ \square =,\textrm{ or } +} (k_1 \square \cdots \square k_r) \in \cR.
\end{align*}
\begin{lem} \label{111}
For a non-empty index $\bk=(k_1, \ldots, k_r)$, we have
\[
\sum_{m=0}^{r-1} (-1)^m 
\sum_{ \substack{\square =,\textrm{ or } + \\ \#(+)=m} } 
(k_1\square \cdots \square k_r)^\star
=(k_1,\dots,k_r).
\]
\end{lem}
\begin{proof}
We prove this lemma by induction on $r$. 
The case $r=1$ is trivial. 
In the general case, by the induction hypothesis, we have 
\begin{align*}
\textrm{L.H.S.} 
&=\sum_{m=0}^{r-2} (-1)^m 
\sum_{ \substack{\square =,\textrm{ or } + \\ \#(+)=m} } 
(k_1,k_2\square \cdots \square k_r)^\star 
-\sum_{m=0}^{r-2} (-1)^m 
\sum_{ \substack{\square =,\textrm{ or } + \\ \#(+)=m} } 
(k_1+k_2\square \cdots \square k_r)^\star \\
&=\sum_{ \triangle =,\textrm{ or } + } 
(k_1\triangle k_2,\dots,k_r) 
-(k_1+k_2,k_3,\dots,k_r) \\
&=\textrm{R.H.S.} \qedhere
\end{align*}
\end{proof}
Hereafter, we want to calculate a sum over an index set such like
\begin{align*}
\{(k_{i+1}, \ldots, k_r, k_1, \ldots, k_i) \mid 1 \le i \le r\}.
\end{align*}
However, same indices may appear more than once in such an expression. 
To treat such overlaps properly, we introduce the following abuse of notation. 
For $M:=\sum_{\bk} c_{\bk} \cdot \bk \in \cR$, we set
\begin{align*}
\sum_{\bk \in M}f(\bk):=\sum_{\bk} c_{\bk} \cdot f(\bk).
\end{align*}
For example, for $M=(2,3)+2(5, 7)$, we have $\sum_{(k_1, k_2)\in M}(k_1, k_2+1)=(2,4)+2(5,8)$.

For a non-empty index $\boldsymbol{k}=(k_{1},\dots,k_{r})$ and $m\in\mathbb{Z}_{\ge0}$ with $m \le r-1$, set 
\[
S_m(\boldsymbol{k})
:=
\sum_{ \substack{\square =,\textrm{ or } + \\ \#(+)=m} } 
(k_{1}\square \cdots \square k_r \square),
\]
where
\begin{align*}
(k_1\square_1 \cdots \square_{r-1} k_r \square_r):=(k_{j+1}\square_{j+1} \cdots \square_{r-1} k_r \square_r k_1 \square_1 \cdots \square_{j-1} k_j)
\end{align*}
with $1 \le j \le r$ such that $\square_j=\ , \ $. 
This definition depends on the choice of $j$, but hereafter we only consider sums such that this ambiguity does not concern.
We note that all indices appearing in $S_m(\boldsymbol{k})$ have the same depth$(=r-m)$.

The following lemma is obvious so we omit the proof.
\begin{lem} \label{112}
For a non-empty index $\boldsymbol{k}=(k_{1},\dots,k_{r})$ and $m\in\mathbb{Z}_{\ge0}$ with $m\le r-1$, we have  
\begin{align*}
\sum_{i=1}^{r-m} \sum_{(l_1,\dots,l_{r-m})\in S_m(\boldsymbol{k})}
   (l_{i+1},\dots,l_{r-m},l_1,\dots,l_i)
=\sum_{ \substack{\square =,\textrm{ or } + \\ \#(+)=m} } 
\sum_{i=1}^{r} (k_{i+1}\square \cdots \square k_r\square k_1 \square \cdots \square k_i).
\end{align*}
\end{lem}
\begin{rem} \label{rem}
Similar to Lemma \ref{112}, we also have
\begin{align*}
\sum_{i=1}^{r-m} 
\sum_{(l_1,\dots,l_{r-m})\in S_m(\boldsymbol{k})}
(l_{i+1},\dots,l_{r-m},l_1,\dots,l_i)^\star
=\sum_{ \substack{\square =,\textrm{ or } + \\ \#(+)=m} } 
\sum_{i=1}^{r} (k_{i+1}\square \cdots \square k_r\square k_1 \square \cdots \square k_i)^\star.
\end{align*}
\end{rem}

\subsection{Proof of the theorem}
To prove Theorem \ref{CSF for tSMZV} from Theorem \ref{cycShat}, we need Propositions \ref{prop1} and \ref{prop2}.  
\begin{prop}\label{prop1}
For a non-empty index $\bk=(k_1, \ldots ,k_r)$ and a non-negative integer $j$, we have
\begin{align*}
&\sum_{i=1}^{r} \bigl\{(j+1+k_i,k_{i+1},\dots,k_r,k_1,\dots,k_{i-1}) +(j+1,k_{i},\dots,k_r,k_1,\dots, k_{i-1})\bigr\} \\
&=\sum_{m=0}^{r-1} (-1)^m 
\sum_{i=1}^{r-m} 
\sum_{(l_1,\dots,l_{r-m})\in S_m(\boldsymbol{k})}
(j+1,l_{i+1},\dots,l_{r-m},l_1,\dots,l_i)^\star. 
\end{align*}
\end{prop}
\begin{proof}
By Lemma \ref{111}, we have
\begin{align*}
&(j+1+k_{i},k_{i+1},\dots,k_r, k_1,\dots,k_{i-1}) +(j+1 ,k_{i},\dots,k_r,k_1,\dots, k_{i-1}) \\
&=\sum_{m=0}^{r-1} (-1)^m 
\sum_{ \substack{\square =,\textrm{ or } + \\ \#(+)=m} } 
(j+1,k_{i}\square \cdots \square k_r\square k_1 \square \cdots \square k_{i-1})^\star. 
\end{align*} 
Then, by Remark \ref{rem}, we have 
\begin{align*}
\textrm{L.H.S.} 
&=\sum_{m=0}^{r-1} (-1)^m 
\sum_{ \substack{\square =,\textrm{ or } + \\ \#(+)=m} }
\sum_{i=1}^{r} 
(j+1, k_{i}\square \cdots \square k_r\square k_1 \square \cdots \square k_{i-1})^\star \\ 
&=\textrm{R.H.S.} \qedhere
\end{align*}
\end{proof}

\begin{prop}\label{prop2}
For a positive integer $k$ and an index $\bk=(k_1, \ldots, k_r)$ with $\wt(\bk)=k$, we have
\begin{align*}
&\sum_{m=0}^{r-1}(-1)^m\sum_{i=1}^{r-m}\sum_{(l_1, \ldots, l_{r-m}) \in S_m(\bk)}\sum_{j=0}^{l_i-1}
(j+1, l_{i+1}, \ldots, l_{r-m}, l_1, \ldots, l_{i-1}, l_i-j)^{\star}\\
&=\sum_{i=1}^r\sum_{j=0}^{k_i-1}(j+1, k_{i+1}, \ldots, k_r, k_1, \ldots, k_{i-1}, k_i-j)-(-1)^rk \cdot (k+1)^{\star}.
\end{align*}
\end{prop}

\begin{proof}
Set $\alpha(\bk):=\sum_{i=1}^r(k_{i+1}, \ldots, k_r, k_1, \ldots, k_i)$. Then we have
\begin{align*}
&\textrm{L.H.S.} \\
&=\sum_{m=0}^{r-1}(-1)^m\sum_{(k'_1, \ldots, k'_r) \in \alpha(\bk)}\sum_{i=1}^r
\sum_{j=0}^{k'_1+\cdots+k'_i-1}\sum_{ \substack{\square =,\textrm{ or } + \\ \#(+)=m-(i-1)} }
(j+1, k'_{i+1}\square \cdots \square k'_r, k'_1+\cdots+k'_i-j)^{\star}\\
&=\sum_{m=0}^{r-1}(-1)^m\sum_{(k'_1, \ldots, k'_r) \in \alpha(\bk)}\sum_{i=1}^r
\sum_{s=1}^i\sum_{j=0}^{k'_s-1}\sum_{ \substack{\square =,\textrm{ or } + \\ \#(+)=m-(i-1)} }\\
&\qquad \qquad \qquad \qquad
(j+1+k'_{s+1}+\cdots+k'_i, k'_{i+1}\square \cdots \square k'_r, k'_1+\cdots+k'_s-j)^{\star}\\
&=\sum_{m=0}^{r-1}(-1)^m\sum_{(k'_1, \ldots, k'_r) \in \alpha(\bk)}\sum_{0 \le a \le b<r}
\sum_{j=0}^{k'_r-1}\sum_{ \substack{\square =,\textrm{ or } + \\ \#(+)=m-r-a+b+2} }\\
&\qquad \qquad \qquad 
(j+1+k'_{1}+\cdots+k'_{a}, k'_{a+1}\square \cdots \square k'_{b}, k'_{b+1}+\cdots+k'_{r}-j)^{\star}\\
&=\sum_{m=0}^{r-1}(-1)^m\sum_{(k'_1, \ldots, k'_r) \in \alpha(\bk)}
\sum_{j=0}^{k'_r-1}\sum_{ \substack{\square =,\textrm{ or } + \\ \#(+)=m} }
(j+1 \square k'_1 \square \cdots \square k'_{r-1} \square k'_r-j)^{\star}\\
&=\sum_{(k'_1, \ldots, k'_r) \in \alpha(\bk)}\sum_{j=0}^{k'_r-1}
\bigl\{(j+1, k'_1, \ldots, k'_{r-1}, k'_r-j)-(-1)^r(j+1+k'_1+\cdots+k'_r-j)\bigr\}\\
&=\textrm{R.H.S.},
\end{align*}
which completes the proof.
\end{proof}

\begin{prop}\label{prop3}
For a non-empty index $\bk=(k_1, \ldots, k_r)$, we have
\begin{align*}
&\sum_{m=0}^{r-1}(-1)^m \sum_{(l_1, \ldots, l_{r-m}) \in S_m(\bk)}\sum_{i=1}^{r-m}
(l_i, \ldots, l_{r-m}, l_1, \ldots, l_{i-1}, 1)^{\star}\\
&=
\sum_{i=1}^r\bigl\{(k_{i+1}, \ldots, k_r, k_1, \ldots, k_i, 1)+(k_{i+1}, \ldots, k_r, k_1, \ldots, k_{i-1}, k_i+1)\bigr\}.
\end{align*}
\end{prop}
\begin{proof}
By Lemma \ref{112} and an easy calculation, we have
\begin{align*}
\textrm{L.H.S.}
&=\sum_{m=0}^{r-1}(-1)^m\sum_{i=1}^{r}\sum_{ \substack{\square =,\textrm{ or } + \\ \#(+)=m} }
(k_{i+1}\square \cdots \square k_r \square k_1 \square \cdots \square k_i, 1)^{\star}\\
&=\sum_{i=1}^r\sum_{m=0}^r(-1)^m \sum_{ \substack{\square =,\textrm{ or } + \\ \#(+)=m} }
(k_{i+1} \square \cdots k_r \square k_1 \square \cdots \square k_i \square 1)^{\star}\\
&\quad+\sum_{i=1}^r \sum_{m=0}^{r-1} (-1)^m \sum_{ \substack{\square =,\textrm{ or } + \\ \#(+)=m} }
(k_{i+1} \square \cdots k_r \square k_1 \square \cdots \square k_{i-1}\square k_i+1)^{\star}.
\end{align*}
By Lemma \ref{111}, this coincides with
\begin{align*}
\sum_{i=1}^r\bigl\{(k_{i+1}, \ldots, k_r, k_1, \ldots, k_i, 1)+(k_{i+1}, \ldots, k_r, k_1, \ldots, k_{i-1}, k_i+1)\bigr\}=\textrm{R.H.S.},
\end{align*}
which completes the proof.
\end{proof} 
\begin{proof}[Proof of Theorem \ref{CSF for tSMZV}]
For an index $\bk=(k_1, \ldots, k_r)$, set
\begin{align*}
&\CSF_{\widehat{\cS}}(\bk)\\
&=\sum_{i=1}^r \sum_{j=0}^{k_i-2} \zeta_{\widehat{\cS}} (j+1, k_{i+1}, \ldots, k_r, k_1, \ldots, k_{i-1}, k_i-j) \\
 &\quad -\sum_{i=1}^r \sum_{j=0}^\infty 
  \left(
   \zeta_{\widehat{\cS}} (k_i+j+1, k_{i+1},  \ldots, k_r, k_1, \ldots, k_{i-1})
   +\zeta_{\widehat{\cS}} (j+1, k_{i+1}, \ldots, k_r, k_1, \ldots, k_i) 
  \right) t^j \\
  &\quad -\sum_{i=1}^r \zeta_{\widehat{\cS}} (k_{i+1}, \ldots, k_r, k_1, \ldots,k_{i-1},  k_i+1)
\end{align*}
and 
\begin{align*}
&\CSF^{\star}_{\widehat{\cS}}(\bk)\\
&=\sum_{i=1}^{r}\sum_{j=0}^{k_i-2}
\zeta^{\star}_{\widehat{\cS}}(j+1, k_{i+1}, \ldots, k_r, k_1, \ldots, k_{i-1}, k_i-j)\\
&\quad -\sum_{i=1}^{r}\sum_{j=0}^{\infty}
\zeta^{\star}_{\widehat{\cS}}(j+1, k_{i+1}, \ldots, k_r, k_1, \ldots, k_i)t^j
-\wt(\bk)\zeta^{\star}_{\widehat{\cS}}(\wt(\bk)+1).
\end{align*}
Note that by Theorem \ref{cycShat}, we have $\CSF^{\star}_{\widehat{\cS}}(\bk)=0$ for any non-empty index $\bk$. Therefore, by Propositions \ref{prop1}, \ref{prop2}, and \ref{prop3}, we see that
\begin{align*}
\CSF_{\widehat{\cS}}(\bk)=\sum_{m=0}^{r-1}(-1)^m \sum_{(l_1, \ldots, l_{r-m}) \in S_m(\bk)}
\CSF^{\star}_{\widehat{\cS}}(l_1, \ldots, l_{r-m})=0,
\end{align*} 
which completes the proof of Theorem \ref{CSF for tSMZV}.
\end{proof}

\section*{Acknowledgement}
The authors are grateful to Prof. Shuji Yamamoto for his valuable comments to the proof of Theorem \ref{main1}. The authors are also grateful to Dr. David Jarossay for informing us about his work \cite{Jar19}.

\end{document}